\documentclass{amsart}
\usepackage{amssymb,latexsym}
\newtheorem{theorem}{Theorem}[section]

\newtheorem{definition}[theorem]{Definition}

\newtheorem{lemma}[theorem]{Lemma}

\def\E{\mathbb{E}}

\def\R{\mathbb{R}}

\title{Random Walks and Mixed Volumes of Hypersimplices}

\date{\today}

\begin{document}
\begin{abstract}
Below is a method for relating a mixed volume computation for polytopes sharing many facet directions to a symmetric random walk.  The example of permutahedra and particularly hypersimplices is expanded upon.  
\end{abstract}
\maketitle

\section{Introduction}
Below is a method for relating a mixed volume computation for polytopes sharing many facet directions to a symmetric random walk.  The example of permutahedra and particularly hypersimplices is expanded upon.

\section{Setup}
A family of polytopes is obtained by fixing a finite set $H$ to index the possible facets, a subspace $X$ of $\R^H$ orthogonal to $\vec{1}=(1,\ldots,1)$ to contain the polytopes and a convex subset $Y$ of $\R^H$ orthogonal to and containing $\vec{1}$ to index the polytopes.  

Write $\R_+$ for the nonnegative real numbers and $F(y)=X\cap (\R_+^H-y)=\{x\in X|x_h\geq -y_h\forall h\in H\}$ for the polytope in $X$ indexed by $y\in Y$.  
Note that $F$ is superaffine: for every $a\in[0,1]$, $y,\hat{y}\in Y$ there is $aF(y)+(1-a)F(\hat{y})\subseteq F(ay+(1-a)\hat{y})$ where the left side involves the Minkowski sum of polytopes.  Typically these polytopes are not the same.  
Faces of the polytopes $F(y)$ will be obtained in two ways. 
 
For every $A\subseteq H$ the face of $F(y)$ supported by $A$ is $F_A(y)=X\cap ((\R_+^{H-A}\times \{0\}^A)-y)=\{x\in X|x_h\geq-y_h\forall h\in H, x_a\geq -y_a\forall a\in A\}$. Write $Y_A=\{y\in Y|F_A(y)\not=\emptyset\}=\pi_1(\{(y,x)\in Y\times X|x_h\geq-y_h\forall h\in H, x_a\geq -y_a\forall a\in A\}$ for the polytope indexing polytopes with faces supported by $A$, $F_A=X\cap(\R^{H-A}\times\{0\}^A)=\{x\in X|x_a=0 \forall a\in A\}$ and $d_A=\dim(F_A)$.  Note that $F_A(y)\subseteq F_A-y$ a translate of $F_A$.

Write $SX^*$ for the unit sphere in the dual space to $X$.  For every $g\in SX^*$ the face of $F(y)$ 
with support direction $g$ is $F_g(y)=F(y)\cap g^{-1}(m)$ where $m=g(F(y))=\hbox{max}_{x\in F(y)}g(x)$ is the height of the polytope $F(y)$ in the $g$ direction.  
To relate these set $S(g,y)\subseteq H$ to be maximal with $F_{S(g,y)}(y)=F_g(y)$ so all coordinates of $S(g,y)$ are perpendicular to the affine span of $F_g(y)$.  Set $C_A=\{g\in SX^*|\exists y\in Y, S(g,y)=A\}$ the support directions for faces with support $A$ so that if $g\in C_A$ and $x\in F_A(y)$ then $g(F(y))=g(x)$.  

The function of interest will be the mixed volume of $\{F(y_i)\}_{i\in I}$ if the $y_i$ are points in $Y$, which is denoted by $V_I(F(y.))$. 
Recall that this means that for $\{\lambda_i\}$ positive real parameters, $\hbox{Vol}_{|I|}(\sum_i\lambda_iF(y_i))=\sum_{J\in \left({I\choose |I|}\right)}\lambda^Jv_J$ where the sum of polytopes in $\R^n$ is the Minkowski sum and $V_I(F(y.))={v_I\over |I|!}$ is the fully mixed (or equivalently, multiplicity free) coefficient. 

\section{Theorems}
\begin{lemma} If $i\in I$ then 
$$V_I(F(y.))=\sum_{g\in SX^*}g(F(y_i))V_{I-\{i\}}(F_g(y.)).$$  
\end{lemma}
\noindent This is a sum rather than a more general integral since the volume term is nonzero for only a finite set of directions $g$.  See, for instance, the integral over the mixed surface function in [1].

\begin{lemma}
If $g\in C_A$ then the restriction of $g(F(\cdot))$ to $Y_A$ is also the restriction of a linear function in $(\R^H)^*$ to $Y_A$.  
\end{lemma}
\begin{proof}
If $a\in[0,1]$ and $y,\hat{y}\in Y_A$ and $x\in F_A(y)$ and $\hat{x}\in F_A(\hat{y})$ then $g(F(ay+(1-a)\hat{y}))=g(ax+(1-a)\hat{x})=ag(F(y))+(1-a)g(F(\hat{y}))$.  If the affine span of $Y_A$ does not contain $0$ this is sufficient and otherwise it follows similarly that $g(F(ay))=ag(F(y))$.  
\end{proof}

\begin{lemma}
If there is $J\subseteq I$ with $\dim(\sum_{j\in J}F(y_j))<|J|$ then $V_I(F(y.))=0$.  
\end{lemma}
\begin{proof}
This is immediate from the definition of mixed volume.  
\end{proof}

\begin{lemma}
If $\{U_i\}_{i\in I}$ are subsets of $Y$ and for every $A\subseteq H$ either 
(first case) every $i$ has $U_i\cap Y_A\in\{\emptyset, U_i\}$ or (second case) there is $B\supseteq A$ 
with $|\{i|U_i\subseteq Y_B\}|>d_B$ then the restriction of $V_{I}(F(\cdot))$ to
$\prod_iU_i$ is multilinear so there are $L_i\in(\R^H)^*$ with $V_{I}(F(y.))=\prod_iL_i(y_i)$ if every $y_i\in U_i$.    
\end{lemma}
Call such a product subset of $Y^{I}$ a zone. 

\begin{proof}
Fix $g\in SX^*$, $y.\in \prod_iU_i$ and $i\in I$.  
Set $A=S(g,y_i)$ so that $g\in C_A$.  By the hypotheses, either $U_i\subseteq Y_A$ or there is $B\supseteq A$ with $|\{j\not= i|U_j\subseteq Y_B\}|> d_B$.  This is because $y_i\in U_i\cap Y_A\not=\emptyset$ and $Y_B\subseteq Y_A$ so if the first case does not hold then $U_i\not\subseteq Y_B$.  

Thus for each $g$, $i$ and $y.\in \prod_{j\not= i}U_j$ the function $$(*)=g(F(y_i))V_{I-\{i\}}(F_g(y.))$$ is a linear function of $y_i$ in $U_i$.  In the first case this is the content of Lemma 3.2.  In the second case $(*)=0$ from Lemma 3.3 with $J=\{j\not=i|U_j\subseteq Y_B\}$.  The lemma applies since every $y\in U_j\subseteq Y_B$ has $F_g(y)\subseteq F_B(y)\subseteq F_B-y$ so that $\sum_{j\in J}F(y_j)$ is contained in a translate of $F_B$ and has dimension at most $d_B<|J|$.  
\end{proof}

Write $\pi_{\hat{i}}:Y^{I}\rightarrow Y^{I-\{i\}}$ for the projection.  
\begin {definition}
The vector space of irrelevant signed measures on $Y^{I}$ is generated by signed measures $\mu$ with $\hbox{supp}(\mu)\subseteq \prod_jU_j$ some zone and some $i$ so that for every measurable $R\subseteq Y^{I-\{i\}}$ there is $\E((\pi_i)_*(f_R\mu))=0\in\R^H$ where $f_R(z)=1$ if $\pi_{\hat{i}}(z)\in R$ and $f_R(z)=0$ otherwise.  
\end{definition}

\begin{lemma}
If $\mu$ is irrelevant then $\int_{y.\in Y^{I}} V_{I}(F(y.))d\mu(y.)=0$.  
\end{lemma}
\begin{proof}
This follows from Lemma 3.4:  Assume that $\mu$ is a generator of irrelevant measures with $i$ and $\{U_j\}$ as in Definition 3.5.  
$$\int_{y.\in Y^{I}} V_{I}(F(y.))d\mu(y.)=\int_{y.\in \prod_jU_j} V_{I}(F(y.))d\mu(y.)$$ 
$$=\int_{y.\in \prod_jU_j} \prod_jL_j(y_j)d\mu(y.)=\lim_{\{R_\alpha\}}\sum_\alpha\int_{y.\in R_\alpha\times U_i} \prod_jL_j(y_j)d\mu(y.)$$ 
$$=\lim_{\{R_\alpha\}}\sum_\alpha\prod_{j\not= i}L_j(r_{\alpha j})\int_{y.\in R_\alpha\times U_i} L_i(y_i)d\mu(y.)$$
$$=\lim_{\{R_\alpha\}}\sum_\alpha \prod_{j\not= i}L_j(r_{\alpha j})L_i(\E((\pi_i)_*(f_{R_\alpha}\mu)))=0.$$
Here the limits are over partitions $\{R_\alpha\}$ of $\prod_{j\not=i}U_j$ and $r_\alpha\in R_\alpha$ so that $r_{\alpha j}\in U_j$.  
\end{proof}

Lemma 3.6 will be used to compute mixed volumes as follows.  
Note that if one probability measure on $Y^{I}$ is obtained from another by allowing one coordinate to perform a symmetric random walk within some zone then the difference between them will be an irrelevant signed measure.  
If $y.\in Y^{I}$ then $V_{I}(F(y.))$ is the integral of 
$V_{I}(F(\cdot))$ over the delta probability measure supported at the configuration 
$y.$ of $|I|$ points in $Y$.  Assume $y.$ is contained in some zone 
$\prod_iU_i$.  Construct a new measure by allowing 
one point $y_i$ to move within $U_i$ via some symmetric process until it 
reaches a boundary point of $U_i$.  This new probability measure will differ 
from the original one by an irrelevant signed measure and hence the integral of 
the volume function will be the same desired value.  This process can then be 
continued making use of the fact that the points at which the new measure is 
supported will all be contained in zones $\prod_jV_j$ for which $\dim(V_i)<\dim(U_i)$, 
while the other regions $V_j\supseteq U_j$ may have become larger.   
In useful cases this will terminate with a measure for which the volume 
integral is easy to compute and the work will have been transferred to 
following the random process to get the final measure.  

\section{Example: Permutahedra}
Permutahedra are the convex hull of a single symmetric group orbit of a point 
in $\R^{n}$ with the defining (trivial plus irreducible) representation of the symmetric group $\Sigma_n$.  
These can have anywhere from $1$ to $n-1$ orbits of facets under $\Sigma_n$.  

For the above setup fix a positive integer $n$ and $R\subseteq [n-1]=\{1,\ldots ,n-1\}$.  
Take $H=\{T\subseteq [n]||T|\in R\}$, $X=\{\sum_{i\in[n]}\delta_i|\sum_ix_i=0\}\subseteq \R^H$ 
where $\delta_i(T)=1$ if $i\in T$ and $0$ otherwise.  Take $Y\subseteq \R^H$ 
to be the simplex which is the convex hull of 
$\{{|H|\over {n\choose r}}\delta_r\}_{r\in R}$ where $\delta_r(T)=1$ if $r=|T|$ and $0$ otherwise.
The polytopes $F(y)$ will be permutahedra with (generically) $|R|$ orbits of facets.  Hypersimplices are the case $R=\{1,n-1\}$ and typically have two facet orbits.  

The cases (such as hypersimplices) in which $|R|=2$ and hence $Y$ is an interval 
are particularly nice.  

Fix $R=\{r,s\}$ with $r<s$ and a linear homeomorphism $j:[0,1]\rightarrow Y$ from the interval to the configuration space $Y$ with $j(0)={|H|\over{n\choose r}}\delta_r$.  

\begin{lemma}
If $\prod_{i\in [n]}[u_{i0},u_{i1}]\subseteq [0,1]^n$ satisfies the following four conditions then $\prod_ij([u_{i0},u_{i1}])\subseteq Y^n$ is a zone and every maximal zone occurs in this way.  

1) $\{u_{ie}\}\subseteq\{p_t\}$ with $t\in\{0\}\cup[r,s]\cup\{n\}$ and $p_t={t{n-1\choose s}\over t{n-1\choose s}+(n-t){n-1\choose r-1}}$.  

2) If $u_{i0}<p_t<u_{i1}$ with $r<t<s$ then there is some $j$ with $u_{j0}=u_{j1}=p_t$.  

3) If $u_{i0}<p_r<u_{i1}$ then $|\{j|u_{j0}=u_{j1}=p_t\}|\geq r$.  

4) If $u_{i0}<p_s<u_{i1}$ then $|\{j|u_{j0}=u_{j1}=p_s\}|\geq n-s$.  
\end{lemma}
\begin{proofsketch}
If $A\subseteq H={[n]\choose r}\cup{{n}\choose s}$ write $A_r=A\cap {[n]\choose r}$ and $A_s=A\cap {[n]\choose s}$.

If $Y_A=j([p,q])$ then there is some $Y_{S(g,y)}=\{j(p)\}$.  

If $A=S(g,y)$ and $Y_A=\{y_A\}=\{j(p_A)\}$ then there is $K\subseteq [n]$ 
with
 $A_r=\{R\in {[n]\choose r}|R\subseteq K\}$ and  $A_s=\{R\in {[n]\choose s}|S\supseteq K\}$.  Write $A_K=A$.  

Compute that $x_K\in F_{A_K}(j(p_{|K|}))$ where $x_K=a\sum_{i\in K}\delta_i-b\sum_{j\ \not\in K}\delta_j$ for appropriate constants $a$ and $b$.  

If $r<|K|<s$ then $\{x_K\}= F_{A_K}(j(p_{|K|}))$.  
If $r=|K|$ and $K'\subseteq K$ then $x_{K'}\in F_{A_K}(j(p_{|K|}))$.  
If $|K|=s$ and $K'\supseteq K$ then $x_{K'}\in F_{A_K}(j(p_{|K|}))$.  

These points span the faces $F_{A_K}(j(p_{|K|}))$ allowing the dimension computations: $d(A_K)=0$ if $r<|K|<s$ while $d(A_K)=r-1$ if $r=|K|$ and $d(A_K)=n-s-1$ if $|K|=s$.

This suffices to work out the asserted zones.  
 \end{proofsketch}

 The random walk procedure with $R=\{r,s\}$ is thus a walk of $n$ particles on the interval $[0,1]$ in which the particles $p_t$ with $r<t<s$ absorb the first particle to reach them and are then invisible, while $p_r$ and $p_s$ each absorb the first several particles to reach them.  The number of particles is the same as the number of possible interior absorptions so that the random walk procedure will end with a measure supported on configurations for which some particle is at an endpoint (which will make the mixed volume $0$) or those which are permutations of $\{p_r\}^r\times\prod_{t=r+1}^{s-1}\{p_t\}\times\{p_s\}^{n-s}$.  Call this particular mixed volume $V$ and compute it separately.  Now any mixed volume is $V$ times the probability that the system ends in a state of this last form.  

 Sometimes this probability is fast to compute.

 Bunched at one end:  For convenience write $q_t=p_t$ if $t\in[r,s]$, $q_t=p_r$ if $t\leq r$ and $q_t=p_s$ if $t\geq s$.  
 \begin{lemma}
 If for every $t$, $y_t\leq q_t$ then $V_{[n]}(F(y.))=V\prod_{t\in[n]} {y_t\over q_t}$.  
 Similarly, if for every $t$, $y_t\geq q_t$ then $V_{[n]}(F(y.))=V\prod_{t\in[n]} {1-y_t\over q_t}$.  
 \end{lemma}
 \begin{proof}
 Consider the sequence of measures obtained by moving the objects in order until the $t$th particle reaches either $q_t$ (which occurs with probability ${y_t\over q_t}$) or $0$.  
 \end{proof}

 Bunched somewhere:

 If $R=\{ r,s\}$ and there are $r<a<b<s$ with every $a\leq t\leq b$ having $p_t\in\{y_i\}$ and every $y_i\in [p_a,p_b]$ then there is a fast algorithm to compute $V_{[n]}(F(y_.))$ by walking one point at a time.  While in the previous case there was only one nonzero state to keep track of the probability for at each step, there are now around $n$.  A similar process will be polynomial time for any fixed number of bunches.  

 This is particularly clean for $R=\{1,n-1\}$ (hypersimplices) since the points $p_i={i\over n}$ are evenly spaced in $[0,1]$.  If $a\in[0,1]$ and $k\in\R$ write $a+_1k$ for $a+k$ mod-$1$ or the fractional part of the sum.  
 \begin{lemma} For every $a_1,\ldots ,a_n\in[0,1]$ there is $$\sum_{k=0}^{k=n-1}V_{[n]}(F(j(a_i+_1{k\over n})))=1.$$
\end{lemma}
\vskip10pt
\noindent {\bf Example 1:}(three dimensional hypersimplices)
\newline Set $n=4$, $R=\{1,3\}$ so $H=\{\{1\},\ldots ,\{4\},\{1,2,3\},\ldots,\{2,3,4\}\}$, $X$ is a copy of $\R^4$ spanned by $x_1=(-{3\over 4},{1\over 4},{1\over 4},{1\over 4},-{1\over 4},-{1\over 4},-{1\over 4},{3\over 4})$, $x_2=({1\over 4},-{3\over 4},{1\over 4},{1\over 4},-{1\over 4},-{1\over 4},{3\over 4},-{1\over 4})$, $x_3=({1\over 4},{1\over 4},-{3\over 4},{1\over 4},-{1\over 4},{3\over 4},-{1\over 4},-{1\over 4})$, $x_4=({1\over 4},{1\over 4},{1\over 4},-{3\over 4},{3\over 4},-{1\over 4},-{1\over 4},-{1\over 4})$, in $\R^8$ while $Y$ is the convex hull of $y_1=(2,2,2,2,0,0,0,0)$ and $y_3=(0,0,0,0,2,2,2,2)$ in $\R^8$ with the linear homeomorphism \newline $j(i)=(2i,2i,2i,2i,2(1-i),2(1-i),2(1-i),2(1-i))$ of $[0,1]$ to $Y$.  
Consider the cases in which every $4u_i=4j^{-1}(y_i)$ is an integer with $1\leq 4u_i\leq 4u_{i+1}\leq 3$ and write $y_.=j((u_1,u_2,u_3))$.  There are ten such $y_.$.  
Note that $F(j(0))$ and $F(j(1))$ are points and hence any mixed volume involving either one is zero, $F(j({1\over 4}))$ and $F(j({3\over 4}))$ are (dual) tetrahedra and $F(j({2\over 4}))$ is an octahedron.
A direct computation gives $V=32\sqrt{2}$.  
The Lemma 4.2 gives: 
\newline $V_{[3]}(j(({1\over 4},{1\over 4},{1\over 4})))=V{1\over 1}{1\over 2}{1\over 3}=V{1\over 6}$
\newline $V_{[3]}(j(({1\over 4},{1\over 4},{2\over 4})))=V{1\over 1}{1\over 2}{2\over 3}=V{2\over 6}$
\newline $V_{[3]}(j(({1\over 4},{1\over 4},{3\over 4})))=V{1\over 1}{1\over 2}{3\over 3}=V{3\over 6}$
\newline $V_{[3]}(j(({1\over 4},{2\over 4},{2\over 4})))=V{1\over 1}{2\over 2}{2\over 3}=V{4\over 6}$
\newline $V_{[3]}(j(({1\over 4},{2\over 4},{3\over 4})))=V{1\over 1}{2\over 2}{3\over 3}=V{6\over 6}$ (as per the normalization) 
\newline and switching $i$ with $4-i$:
\newline $V_{[3]}(j(({3\over 4},{3\over 4},{3\over 4})))=V{1\over 1}{1\over 2}{1\over 3}=V{1\over 6}$
\newline $V_{[3]}(j(({2\over 4},{3\over 4},{3\over 4})))=V{1\over 1}{1\over 2}{2\over 3}=V{2\over 6}$
\newline $V_{[3]}(j(({1\over 4},{3\over 4},{3\over 4})))=V{1\over 1}{1\over 2}{3\over 3}=V{3\over 6}$
\newline $V_{[3]}(j(({2\over 4},{2\over 4},{3\over 4})))=V{1\over 1}{2\over 2}{2\over 3}=V{4\over 6}$
\newline $V_{[3]}(j(({1\over 4},{2\over 4},{3\over 4})))=V{1\over 1}{2\over 2}{3\over 3}=V{6\over 6}$ (again)
\newline while the only one not fitting Lemma 4.2 is:
\newline $V_{[3]}(j(({2\over 4},{2\over 4},{2\over 4})))=V{1\over 1}{2\over 2}{2\over 3}=V{4\over 6}$.

Lemma 4.3 gives:
\newline $0+V_{[3]}(j(({1\over 4},{1\over 4},{1\over 4})))+V_{[3]}(j(({2\over 4},{2\over 4},{2\over 4})))+V_{[3]}(j(({3\over 4},{3\over 4},{3\over 4})))=V,$
\newline $0+V_{[3]}(j(({1\over 4},{1\over 4},{2\over 4})))+V_{[3]}(j(({2\over 4},{2\over 4},{3\over 4})))+0=V,$
\newline $0+V_{[3]}(j(({1\over 4},{1\over 4},{3\over 4})))+0+V_{[3]}(j(({1\over 4},{3\over 4},{3\over 4})))=V,$
\newline $0+V_{[3]}(j(({1\over 4},{2\over 4},{2\over 4})))+V_{[3]}(j(({2\over 4},{3\over 4},{3\over 4})))+0=V$ and 
\newline $0+V_{[3]}(j(({1\over 4},{2\over 4},{3\over 4})))+0+0=V.$

\vskip10pt
\noindent {\bf Example 2:}(three dimensional nonhypersimplices)
\newline Set $n=4$ and $r=\{1,2\}$ so that $|H|=4+6$ and the transitions occur at 
$p_1={1\over 2}$ which absorbs one particle and $p_2={3\over 4}$ which absorbs two with $F(j(p_1))$ again a simplex and $F(j(p_2))$ a cube.  

\section{questions}
\vskip10pt
{\bf Guess:} If $R=\{1,n-1\}$ and every $nv_i\in[0,n]$ is integral then ${(n-1)!\over V}V_{[n-1]}(j(v.))$ is congruent modulo $n$ to $\pm n^n\prod_iv_i$.

This guess is clearly true for the cases in Lemma 4.2.  For th this only leaves $\pm2\cdot2\cdot2$ congruent to $4$ mod $4$ to check.   
  
\vskip10pt
{\bf Question:} Is there a polynomial time algorithm for mixed volumes of hypersimplices?

\bibliographystyle{amsplain}

\begin{thebibliography}{10}

\bibitem{mv}
http://www.encyclopediaofmath.org/index.php?title=Mixed-volume\_theory

\bibitem{c}
Dorian Critoru, \emph{Mixed Volumes of Hypersimplices, Root Systems and Shifted Young Tableaux}, Department of Mathematics, MIT in partial fulfillment of the requirements of the degree of PhD

\end{thebibliography}

\def\cprime{$'$} \def\cprime{$'$}
\providecommand{\bysame}{\leavevmode\hbox to3em{\hrulefill}\thinspace}
\providecommand{\MR}{\relax\ifhmode\unskip\space\fi MR }
\providecommand{\MRhref}[2]{%
  \href{http://www.ams.org/mathscinet-getitem?mr=#1}{#2}
}
\providecommand{\href}[2]{#2}

\vskip20pt
\noindent {\bf Authors:} Eric Babson and Einar Steingrimsson

\end{document}